\documentclass[preprint,12pt,3p]{elsarticle}

\usepackage{amssymb}

\journal{Ars Combinatoria}

\usepackage{booktabs}
\usepackage{blindtext}
\usepackage{graphicx}
\usepackage{amsmath}
\usepackage{algorithm}
\usepackage[noend]{algpseudocode}

\usepackage{graphicx}
\usepackage{xcolor}
\usepackage{geometry} 
\usepackage{array} 
\usepackage{textcomp} 
\usepackage[all,cmtip]{xy} 
\usepackage{varioref}
\usepackage{caption}
\usepackage{subcaption}
\usepackage{epstopdf}

\usepackage{amssymb,amsthm,amsmath,amsfonts,latexsym,tikz}
\newtheorem{thm}{Theorem}[section]
\newtheorem{prop}[thm]{Proposition}
\newtheorem{cor}[thm]{Corollary}
\newtheorem{lem}[thm]{Lemma}
\newtheorem{conj}[thm]{Conjecture}

\makeatletter
\def\BState{\State\hskip-\ALG@thistlm}
\makeatother

\DeclareMathOperator{\diam}{diam}

\DeclareMathOperator{\ddp}{dp}
\DeclareMathOperator{\DP}{DP}
\DeclareMathOperator{\dpp}{dp}
\DeclareMathOperator{\V}{V}

\begin{document}

\begin{frontmatter}

\title{Distance preserving graphs and graph products}

\author[label1]{M. H. Khalifeh}
\address[label1]{Department of Mathematics, Michigan State University, \\
              East Lansing, MI 48824-1027, U.S.A.}
\address[label2]{Department of Computer Science and Engineering,	Michigan State University\\
							East Lansing, MI 48824, U.S.A.}

\author[label1] {Bruce E. Sagan}

\author[label1,label2]{Emad Zahedi\corref{cor1}}

\ead{Zahediem@msu.edu}
\cortext[cor1]{I am corresponding author}

\begin{abstract}
If $G$ is a graph then a subgraph $H$ is {\em isometric} if, for every pair of vertices $u,v$ of $H$, we have $d_H(u,v) = d_G(u,v)$ where $d$ is the distance function. We say a graph $G$ is {\em distance preserving (dp)} if it has an isometric subgraph of every possible order up to the order of $G$. We give a necessary and sufficient condition for the lexicographic product of two graphs to be a dp graph. A graph $G$ is {\em sequentially distance preserving (sdp)} if the vertex set of $G$ can be ordered so that, for all $i\ge1$, deleting the first $i$ vertices in the sequence results in an isometric graph. We show that the Cartesian product of two graphs is sdp if and only if each of them is sdp. In closing, we state a conjecture concerning the Cartesian products of dp graphs. 
\end{abstract}


\end{frontmatter}


\section{Introduction}\label{section:1}
The computational complexity of exploring distance properties of large graphs such as real-world social networks which consist of millions of nodes can be extremely expensive. Recomputing distances in subgraphs of the original graph will add to the cost. One way to avoid this
is to use  subgraphs where the distance between any pair of vertices is the same as in the original graph. Such a subgraph is called {\em isometric}.
Isometric subgraphs come into play in network clustering~\cite{nussbaum2013clustering}.

One family of graphs which has been studied in the literature involving isometric subgraphs is the set of distance-hereditary graphs. A {\em distance-hereditary graph} is a connected graph in which every connected induced subgraph of $G$ is isometric. Distance-hereditary graphs have appeared in various papers~\cite{bandelt1986distance, damiand2001simple, hammer1990completely} since they were first described in an article of Howorka~\cite{howorka1977characterization}. Distance-hereditary graphs are known to be perfect graphs~\cite{golumbic2000clique,d1988distance}.

Another notion using isometric subgraphs is that of a distance preserving graph.  A connected graph is {\em distance preserving}, for which we use the abbreviation dp, if it has an isometric subgraph of every possible order. The definition of a distance-preserving graph is similar to the one for distance-hereditary graphs, but is less restrictive. Because of this, distance-preserving graphs can have a more complex structure than distance-hereditary ones. Distance-preserving graphs have also been studied in the literature. See, for example, \cite{esfahanian2014constructing,NF2,zahedidistance}. 

We will also consider a related notion defined as follows. A connected graph $G$ is {\em sequentially distance preserving (sdp)} if there is some  ordering $v_1,v_2,\dots, v_{|V(G)|}$ of the vertices of $G$ such that the subgraph $G-\{v_i\}_{i=1}^s$ is an isometric subgraph of $G$ for $1\leq s\leq |V(G)|$,~\cite{chepoi1998distance}. Obviously every distance-hereditary graph is sdp and every sdp graph is dp.  

The purpose of this paper is to investigate what happens to the dp and sdp properties when taking products of graphs. Graph products are operations which take two graphs $G$ and $H$ and produce a graph with vertex set $V(G)\times V(H)$ and  certain conditions on the edge set~\cite{imrich2000product}. We consider two kinds of such products, lexicographic product and Cartesian product. Various graph invariants of lexicographic products of graphs have been studied in the literature. See, e.g., \cite{anand2012convex,vcivzek1994chromatic,yang2013connectivity}. The Cartesian product is a well-known graph product, in part because of Vizing's Conjecture~\cite{vizing1963cartesian}, and has been considered by many authors such as~\cite{aurenhammer1992cartesian, caceres2007metric, khalifeh2009some, yousefi2008pi}. 

The outline of this paper is as follows. Section~\ref{section:2} gives full definitions for the main concepts we will need. Section~\ref{section:3} gives a necessary and sufficient condition for the lexicographic product of two graphs to be dp. This condition implies that if $G$ is dp then the lexicographic product of $G$ and any graph $H$ is dp. Moreover, all isometric subgraphs of the lexicographic product of two arbitrary graphs are characterized in this section. In the Section~\ref{section:4}, we will show that the Cartesian product of two graphs is sdp if and only if its factors are. We end with a conjecture about when the Cartesian product of graphs is dp.

\section{Preliminaries }\label{section:2}
In this paper every graph $G=(V,E)$ will be finite, undirected and simple. For ease of notation, we let  $|G|$ be the number of vertices of $G$. A sequence of vertices $u_0,u_1,\dots,u_l$ is a {\em walk of length l} if $u_{i-1}u_i\in E$
for   $1\le i\le l$.  The walk is a {\em path} if the $u_i$ are distinct.
The {\em distance} between vertices $u,v$ in $G,\ d_G(u,v),$ is the minimum length 
of a path connecting  $u$ and $v$. In the case of a disconnected graph $G$, we let $d_G(u,v) = \infty$ when there is no path between $u$ and $v$ in $G$. If the graph $G$ is clear from  context, we will use $d(u,v)$, instead of $d_G(u,v)$. A path $P$ from $u$ to $v$ with length $d(u,v)$ is called a $u$--$v$ {\em geodesic}. 

A subgraph $H$ of a graph $G$ is called an {\em isometric} subgraph, denoted $ H \leq G$, if $d_H(u,v)=d_G(u,v)$ 
for every pair of vertices $u,v \in V(H)$. A connected graph $G$ with $|G|=n$ is called {\em distance preserving (dp)} if it has an $i$-vertex isometric subgraph for every $1 \leq i \leq n$. A connected graph $G$ is called {\em sequentially distance preserving (sdp)} if there is an ordering $u_1,\dots, u_n$ of the vertices of $G$ such that $G - \{u_i\}_{i=1}^s\le G$ for $1\leq s\leq n$.  In this case we say that  $u_1,\dots,u_n$ is an {\em sdp sequence} for $G$.

The {\em lexicographic product} $G[H]$ of graphs $G$ and $H$ is the graph with vertex set $V(G)\times V(H)$ and edge set
$$
		E(G[ H]) =\{(u,x)(v,y)\ |\text{ $uv\in E(G)$, or $xy\in E(H)$ and $u=v$}\}.
$$	
The {\em Cartesian product} of $G$ and $H$ is the graph, denoted $G\ \Box\ H$, on the vertex set $V(G)\times V(H)$  whose edge set is
$$
		E(G\ \Box\ H) =\{(u,x)(v,y)\ |\text{ $uv\in E(G)$ and $x=y$, or $xy\in E(H)$ and $u=v$} \}.
$$
The reader can consult the book of Imrich and Klavzar~\cite{imrich2000product}, for more details about products.

\section{Lexicographic products of graphs}\label{section:3}
In this section we derive a necessary and sufficient condition for a connected graph $G[H]$ to be distance preserving. Furthermore we will find all the isometric subgraphs of $G[H]$. 

We first need a lemma about the distance function in $G[H]$ which is proved by Khalifeh et al.~\cite{khalifeh2008matrix}
\begin{lem}\label{L1}
Suppose $G$ is a graph with $|G|\ge2$ and $H$ is an arbitrary graph.

(a)
Let $G$ be connected.   For  distinct vertices $(u,x)$ and $(v,y)$ in $G[H]$, 
$$
d_{G[H]}\bigl((u,x),(v,y)\bigl) =
	\begin{cases}
		d_G(u,v) & if\ u\neq v,      \\
		2 & if\ u=v,\ xy\notin E(H),  \\ 
		1 & if\ u=v,\   xy\in E(H).   \\ 
	\end{cases}
$$

 (b)
The graph $G[H]$ is connected if and only if $G$ is connected.
\end{lem}
In order to state the main theorem of this section, we need some notation. Let
$$ 
\dpp(G) = \{k\ \big|\text{ $G$ has an isometric subgraph with $k$ vertices}\}.
$$
If $a,b$ are integers  with $a<b$,  then let
$ [a,b]=\{ a,a+1,a+2,\dots,b\} $. 
So a graph $ G $ is dp if and only if  $ \dpp(G)=[1,|G|]$.
Two elements  $a,b \in \dpp(G)$ {\em bound a non-dp interval} if the set of integers $c$ with $a<c<b$ is nonempty and consists only of elements not in $\dpp(G)$.

Finally, the {\em projection} of a subgraph $K$ of $G[H]$, denoted $\pi(K)$,  is the induced subgraph of $G$ whose vertex set is
$$
 \V \big( \pi (K) \big) = \{u\ \big|\text{ $(u,x)$ is a vertex of $K$} \}.
$$

\begin{thm}\label{th1}
Let $G$ be a connected graph with $|G|\ge2$ and $H$ be an arbitrary graph with $|H|=n$. Then 
$$ 
G[H] \text{\  is dp if and only if \  } b\leq an+1 
$$ 
for every pair $a,b \in \dpp(G)$ bounding a non-dp interval.
\end{thm}
\begin{proof}
We claim, for an induced subgraph $K$ of $G[H]$ with $\pi(K)$ having at least two vertices,
\begin{equation}\label{pi(K)}
 \text {$\pi(K) \leq G$ if and only if $K \leq G[H]$.} 
\end{equation} 

To prove the forward direction of the claim, assume that $\pi(K) \leq G$ and consider
distinct vertices $(u, x), (v, y)\in V(K)$. 
If $u\neq v$ then, using the same ideas as in the proof of the first case in Lemma~\ref{L1}(a), we see that 
$d_{\pi(K)}(u,v)=d_K\bigl((u, x), (v, y)\bigl)$. Using $\pi(K) \leq G$ and the lemma itself gives
$$
d_K\bigl((u, x), (v, y)\bigl)=d_{\pi(K)}(u,v)=d_G(u,v)=d_{G[H]}\bigl((u, x), (v, y)\bigl)
$$
as desired.
 If $u = v$ and $xy\not\in E(H)$, then a similar proof shows that $d_K\bigl((u,x),(v,y)\bigl) =2= d_{G[H]}\bigl((u,x),(v,y)\bigl)$.
Finally, if $u = v$ and $xy\in E(H)$, since $K$ is induced we have $d_K\bigl((u,x),(v,y)\bigl) =1= d_{G[H]}\bigl((u,x),(v,y)\bigl)$.

Conversely, if $K\leq G[H]$, then we must show
$$
d_{\pi(K)}(u,v) = d_{G}(u,v) 
$$
for any two distinct vertices $u,v$ in $\pi(K)$. Again using the ideas in the proof of the first case in Lemma~\ref{L1}(a), we see that 
$d_{\pi(K)}(u,v)=d_K\bigl((u, x), (v, y)\bigl)$ for any $x,y\in V(H)$. Using $K\leq G[H]$ and the lemma itself, we have
$$
d_{\pi(K)}(u,v)=d_K\bigl((u, x), (v, y)\bigl)=d_{G[H]}\bigl((u, x), (v, y)\bigl)=d_G(u,v).
$$

To prove the theorem suppose that $|\pi(K)| = c$, $|G|=m$ and $|H|=n$ 
so that $|G[H]|=mn$.
By definition of projection $c\le |K|\le cn$.   
Also  every connected graph with at least two vertices has  isometric subgraphs with one  vertex and with two vertices.  So by  equation~(\ref{pi(K)}), $G[H]$ will be dp if and only if 
$$
\bigcup_{c\in\dpp(G)} [c,cn]=[1,mn].
$$
Since $1,2,m\in\dpp(G)$, the last equality is equivalent to $[a,an]\cup[b,bn]$ being an interval for every pair $a,b\in\dpp(G)$ bounding a non-dp interval.  But this is equivalent to $b\le an+1$.
\end{proof}

The next result is an immediate corollary of the previous theorem.
 
\begin{cor}
If $G$ is $dp$ with $|G|\ge2$ then so is $G[H]$ for any graph $H$.
\end{cor}

Similarly, the next result follows easily from Lemma~\ref{L1} and equation~(\ref{pi(K)}).
 
\begin{cor}
For a connected graph $G$ with $|G|\ge2$ and an induced subgraph $K$ of $G[H]$, 

$$
K\leq G[H] \text{ if and only if }
	\begin{cases}
		\pi(K)\leq G & \text{if}\ |\pi(K)|\ge2, \\
		\diam(K)\le 2 & \text{if}\ |\pi(K)|=1.
	\end{cases}
$$
\end{cor}


\section{Cartesian product graphs}\label{section:4}

We now turn to Cartesian products and the sdp property. We first need some  notation and a few well-known results. 
A {\em removal set} in $G$ is a set of vertices of $G$ whose removal gives an isometric subgraph, let
$$ 
\DP'(G)=\bigl\{ A\subseteq V(G)\,\big|\, G-A \le G\bigl\}\ \ \text{  and }\ \ \ddp'(G)=\bigl\{ |A|\,\big|\, A \in \DP'(G)\bigl\}.
$$

\begin{prop}~\cite{brevsar2008geodetic}\label{proposition:1}
Suppose $G$ and $H$ are graphs,

(a) If $(u,x)$ and $(v,y)$ are vertices of a Cartesian product $G\ \Box\ H $ then
	\begin{align*}
		d_{G\ \Box\ H}\bigl((u,x),(v,y)\bigl)=d_G(u,v)+d_H(x,y).
	\end{align*}

(b) A path $(u_0,x_0)\dots (u_l,x_l)$ is geodesic in $G\ \Box\ H $ if and only if $u_0\dots u_l$ is a geodesic in $G$ after removal of repeated vertices and similarly for
 $x_0\dots x_l$  in $H$.
\end{prop}

Next we consider isometric Cartesian product subgraphs of a Cartesian product graph.
 
\begin{lem}\label{lemma:1}
Suppose $G'$ and $H'$ are nonempty subgraphs of $G$ and $H$ respectively, then
	$G'\ \Box\ H'\leq G\ \Box\ H$ if and only if $G'\leq G$ and $H'\leq H$.
\end{lem}

\begin{proof} For the forward direction using the assumption and proposition~\ref{proposition:1}(a) we have 
\begin{align*}
d_{G'}(u,v) +d_{H'}(x,y)&  = d_{G'\ \Box\ H'}\bigl((u,x),(v,y)\bigl)\\ &
= d_{G\ \Box\ H}\bigl((u,x),(v,y)\bigl)\\ &
 =d_{G}(u,v)+ d_{H}(x,y),
\end{align*}
for every pair of vertices $(u,x),(v,y)\in V(G'\ \Box\ H')$. As any distance in a subgraph is greater than or equal to the corresponding distance in the original graph, we get $d_{G'}(u,v)=d_{G}(u,v)$ and $d_{H'}(x,y) = d_{H}(x,y)$.

Conversely, suppose $G'$ and $H'$ are isometric subgraphs,  by proposition~\ref{proposition:1}(a) we have
\begin{align*}
d_{G'\ \Box\ H'}\bigl((u,x),(v,y)\bigl) &
= d_{G'}(u,v)+d_{H'}(x,y)\\&
= d_{G}(u,v)+d_{H}(x,y)\\ &
 = d_{G\ \Box\ H}\bigl((u,x),(v,y)\bigl),
\end{align*}	
for each pair of vertices $(u,x),(v,y)\in V(G'\ \Box\ H')$. This complete the proof. 
\end{proof}

We now prove a lemma about removal sets of vertices. 

\begin{lem}\label{lemma:2}
For nonemty subsets $A$ and $B$ in the vertex set of graphs $G$ and $H$ respectively,  
$A\times B\in \DP'(G\ \Box\  H)$ if and only if $A\in \DP'(G)$ and $B\in \DP'(H)$.
\end{lem}

\begin{proof}
To prove the forward direction, we show $A\in \DP'(G)$ as $B\in \DP'(H)$ is similar.  
Let $u,v\in V(G-A)$ and $x\in B$.  By  Proposition~\ref{proposition:1}(b), the $(u,x)$--$(v,x)$ geodesics in 
$(G-A)\ \Box\ B$ are the same as the geodesics in $(G\ \Box\ H)-(A\times B)$.
Now using this fact,  Proposition~\ref{proposition:1}(a), and the assumption in this direction
$$
  d_{G-A}(u,v) 
    = d_{(G-A)\ \Box\ B}\bigl((u,x),(v,x)\bigl)
    = d_{(G\ \Box\ H)-(A\times B)}\bigl((u,x),(v,x)\bigl)
  = d_{G\ \Box\ H}\bigl((u,x),(v,x)\bigl),
$$
Finally, applying  Proposition~\ref{proposition:1}(a) again shows  that the last distance equals $d_G(u,v)$
as desired.

To see the backward direction, first note that 
$(G\ \Box\  H) - (A\times B) = ((G-A)\ \Box\  H)\cup (G\ \Box\  (H-B))$.
So it suffices to show that 
$$
d_{(G\ \Box\ H)-(A\times B)}\bigl((u,x),(v,y)\bigl) = d_{G\ \Box\  H}\bigl((u,x),(v,y)\bigl)
$$
for any $(u,x)$ in $(G-A)\ \Box\ H$ and $(v,y)$ in $G\ \Box\ (H-B)$ since Lemma~\ref{lemma:1} takes care of the other possibilities.  Clearly there is a path
$(u,x),\dots ,(u,y)$ with length  $d_H(x,y)$ in $(G-A)\ \Box\ H$, and also
$(u,y),\dots ,(v,y)$ with length  $d_G(u,v)$ in $G\ \Box\ (H-B)$.
The concatenation of these paths is a path from $(u,x)$ to $(v,y)$ in $(G\ \Box\ H)-(A\times B$)
of length $d_G(u,v)+d_H(x,y)=  d_{G\ \Box\  H}\bigl((u,x),(v,y)\bigl)$ and so must be a geodesic.  This concludes the proof.
\end{proof}

We are now in a position to prove the main theorem of this section.

 \begin{thm}\label{thm:2} 
The product $G\ \Box \  H$ is sdp  if and only if both $G$ and $H$ are sdp.
 \end{thm}

\begin{proof} 
For the forward direction, we will prove that $G$ is sdp, the proof for $H$ being similar.  Take an sdp sequence of vertices for $G\ \Box\ H$.  Fix $x\in H$ and consider the subsequence $(u_1,x),(u_2,x),\dots,(u_n,x)$ where 
$n=|G|$.  We claim that $u_1,u_2,\dots,u_n$ is an sdp sequence for $G$.  Indeed, let
$G'=G-\{u_i\}_{i=1}^s$  and let $K'$ be $G\ \Box\ H$ with the vertices through $(u_s,x)$ removed so that
$G'\ \Box\ \{x\}\subseteq K'$. Now if $v,w\in V(G')$ then, by  Proposition~\ref{proposition:1}(b), $P$ is a $v$--$w$ geodesic in $G'$ if and only if $P\ \Box\ \{x\}$ is a $(v,x)$--$(w,x)$ geodesic in $K'$.  From this fact, the sdp property of the original sequence, and Proposition~\ref{proposition:1}(a) we obtain
$$
d_{G'}(v,w)=d_{K'}\bigl((v,x),(w,x)\bigl)=d_{G\ \Box\ H}\bigl((v,x),(w,x)\bigl)=d_G(v,w)
$$
as desired.

For the converse, suppose that if $u_1,\dots,u_n$ and $v_1,\dots,v_m$ are sdp sequences for $G$ and $H$, respectively.  Then it follows easily from  Lemma~\ref{lemma:2} and the transitivity of the isometric subgraph relation that 
$$
(u_1,v_1),\dots,(u_n,v_1),(u_1,v_2),\dots,(u_n,v_2),\dots,(u_1,v_m),\dots,(u_n,v_m)
$$
is an sdp sequence for  $G\ \Box\ H$.
\end{proof}

The relationship between Cartesian product and the dp property seems more delicate.  In particular, we note that
 $G\ \Box\ H$ can be dp even though  $G$ or $H$ may not be.
As an example suppose a graph $G$ consists of the cycle $C_7$ with a pendant edge and $H$ is the path $P_2$.  It is easy to see that $G$ does not have any isometric subgraph of order $5$.  But using Lemma~\ref{lemma:2} one can prove that $G\ \Box \ H$ is dp. Computations suggest  the following conjecture.  

\begin{conj}
If $G$ and $H$ are dp then so is $G\ \Box\  H$.
\end{conj}

\bibliography{references}
\bibliographystyle{abbrv}

\end{document}